\documentclass[11pt,notitlepage,onecolumn]{article}

\usepackage{fullpage}

\usepackage[utf8]{inputenc}
\usepackage{url}            
\usepackage{booktabs}       
\usepackage{amsfonts}       
\usepackage{nicefrac}       
\usepackage{microtype}      
\usepackage{graphicx} 
\usepackage{amsmath}
\usepackage{amsthm}
\usepackage{amssymb}
\usepackage{mathtools}
\usepackage{bbm}
\usepackage[noend]{algorithmic}
\usepackage[ruled,vlined]{algorithm2e}

\usepackage{todonotes}

\newtheorem{theorem}{Theorem}[section]
\newtheorem{lemma}[theorem]{Lemma} 
\newtheorem{corollary}[theorem]{Corollary}
\newtheorem{proposition}[theorem]{Proposition}

\newtheorem{definition}[theorem]{Definition}

\newtheorem{assumption}{Assumption}

\newcommand{\E}{{\mathbb E}}

\newcommand{\eps}{\varepsilon}

\RequirePackage[normalem]{ulem} 
\RequirePackage{color}\definecolor{RED}{rgb}{1,0,0}\definecolor{BLUE}{rgb}{0,0,1} 

\newcommand{\Cov}{\mathrm{Cov}}
\newcommand{\Var}{\mathrm{Var}}
\newcommand{\Ent}{\mathrm{Ent}}

\newcommand{\bone}{\mathbbm{1}}

\begin{document}
\title{Influences in Mixing Measures}
\author{Frederic Koehler\footnote{\texttt{fkoehler@stanford.edu}, Stanford University.} \and Noam Lifshitz\footnote{\texttt{noam.lifshitz@mail.huji.ac.il}, Hebrew University.} \and Dor Minzer\footnote{\texttt{dminzer@mit.edu}, MIT.} \and Elchanan Mossel\footnote{\texttt{elmos@mit,edu}, MIT.}}

\maketitle

\begin{abstract}
The theory of influences in product measures has profound applications in
theoretical computer science, combinatorics, and discrete probability.
This deep theory is intimately connected to functional inequalities and to the Fourier analysis of discrete groups.
Originally, influences of functions were motivated by the study of social choice theory, wherein a Boolean function represents a voting scheme, its inputs represent the votes,
and its output represents the outcome of the elections. Thus, product measures represent a scenario in which the votes of the parties are randomly and independently
distributed, which is often far from the truth in real-life scenarios.

We begin to develop the theory of influences for more general measures under mixing or correlation decay conditions.
More specifically, we prove analogues of the KKL and Talagrand influence theorems for Markov Random Fields on bounded degree graphs with correlation decay.
We show how some of the original applications of the theory of in terms of voting and coalitions extend to general measures with correlation decay.
Our results thus shed light both on voting with correlated voters and on the behavior of general functions of Markov Random Fields (also called ``spin-systems") with correlation decay.
\end{abstract}
\newpage


\section{Introduction}
 Starting with the works of Ben-Or and Linial~\cite{BenorLinial:90} and Kahn, Kalai, and Linial~\cite{KaKaLi:88},
 Analysis of Boolean functions became a major area of research in combinatorics, probability and theoretical computer science.
 It has deep and interesting connections to functional and isoperimetric inequalities, and other important areas in probability and combinatorics. It has deep impact in property testing, hardness of approximation, the theory of voting and the theory of percolation, see e.g.~\cite{ODonnell:14,GarbanSteif:14,Mossel:22}.

 At the technical level this theory crucially relies on:
 \begin{itemize}
 \item Hyper-contractive inequalities that hold for product measures that are not too biased, and
 \item Explicit representations of functions in explicit bases, which correspond to Fourier bases and their generalizations.
 \end{itemize}

 Major recent effort has been devoted to extend the theory to space for which hyper-contractive inequalities do not hold.
 Notably it was shown that a notion of {\em global hypercontraction} holds for such spaces and that this in turn implies many interesting applications~\cite{khot2023pseudorandom,keevash2021global,KLMan,gur2022hypercontractivity,bafna2022hypercontractivity,MZ}.
 In the other direction, extending the theory to spaces that are not highly symmetric and do not have explicit bases remained a major challenge.

Our main contribution in this paper is to prove very general versions of two major theorems of analysis of Boolean functions, the KKL and the Talagrand theorem in the setting of general Gibbs measures on bounded degree graphs with correlation decay.
 The study of such measures is fundamental in statistical physics, graphical models, and in the analysis of Markov chains and spectral independence, see e.g.~\cite{mossel2009rapid,mossel2009hardness,efthymiou2019convergence,anari2021spectral,marton2019logarithmic,stroock1992logarithmic,chen2021optimal,montenegro2006mathematical,DobrushinShlosman:85}. 
 Such measures are known to satisfy the log-Sobolev inequality (equivalently they are hyper-contractive) but do not posses explicit orthogonal bases.

We show how some of the original applications of the theory of influences extend to the new setup: for general voting functions on $n$ voters there exist a voter who influence is $\Omega(\log n / n)$ times the variance. For monotone voting functions there exist a coalition of $O(n / \log n)$ voters who by flipping their votes can control the elections with probability arbitrary close to $1$.




%
%
%
%
%

\section{Definitions and Main Results}
We recall the definition of the Glauber dynamics, log-Sobolev constant, etc. See e.g., \cite{bakry2014analysis,montenegro2006mathematical,van2014probability} for references.

\paragraph{Glauber dynamics.}
Let $\nu$ be a probability distribution on the space $\Sigma^n$ where $\Sigma$ is an arbitrary finite set.
Let $P_i$ be the Markov operator that resamples coordinate $i$ from stationary distribution $\nu$ conditioned on all other coordinates, so that
\[ (P_i f)(x) = \E_{\nu}[f(X) \mid X_{- i} = x_{- i} ],
\]
where $x_{-i}$ is the vector of all coordinates other than $i$.
We will consider the continuous time Glauber dynamics, where a coordinate $i$ are picked according to independent Poisson clocks and are then the coordinate is updated according to $P_i$.
It is well known that this defines a semigroup $H_t$ where $H_t$ is the transition matrix of the configuration from time $0$ to time $t$. We recall that $H$ being a {\em semigroup} means that it satisfies that $H_{s+t} = H_s H_t = H_t H_s$ for all $s$ and $t$.
Moreover, we can write $H_t = e^{t L}$,
where $L$, called the {\em generator}, is given by $L = \sum_i L_i$ and $L_i f = P_i f - f$ so that $L_i^2 = (P_i - I)^2 = -P_i + I = -L_i$.   With this notation, the {\em Dirichlet form} of the Glauber dynamics is defined to be
\[ \mathcal E_{\nu}(f,f) = -\mathbb E_{X \sim \nu}[ f(X) (L f)(X)] = \sum_i \E_{\nu} (L_i f)^2. \]
Each $L_i$ can be thought of as a generalized notion of partial derivative with respect to coordinate $i$, so the Dirichlet form can be viewed as a natural measure of the size of the gradient of the function $f$ (from the perspective of the chosen semigroup).

\paragraph{Log-Sobolev inequality.} We say the Glauber dynamics for $\nu$ satisfy the log-Sobolev inequality with constant $\rho > 0$ if
\[ \rho\, \Ent_{\nu}[f] \le 2 \mathcal E_{\nu}(\sqrt{f}, \sqrt{f}) \]
for all functions $f : \Sigma^n \to \mathbb{R}_{\ge 0}$, where $\Ent_{\nu}[f] = \E_{\nu}[f \log f] - \E_{\nu}[f] \log \E_{\nu}[f]$ is the relative
entropy functional. This is equivalent to the hypercontractivity statement that for all functions $f$, $t \ge 0$, and $p \ge 1 + e^{-2\rho t}$,
\[ \|H_t f\|_{2} \le \|f\|_p \]
where $\|\cdot\|_p$ denotes the $L_p(\nu)$ norm $\|f\|_p = (\E_{\nu} |f|^p)^{1/p}$.

The log-Sobolev inequality implies that the Poincar\'e inequality
\[ \lambda \Var_{\nu}(f) \le \mathcal E_{\nu}(f,f) \]
holds with some constant $\lambda \ge \nu$ and for all functions $f : \Sigma^n \to \mathbb R$. This is
equivalent to the statement that $\Var(H_t f) \le e^{-\lambda t} \Var(f)$ for all such $f$.

\paragraph{Markov property.} We say $\nu$ is a \emph{Markov random field} with respect to a graph $G$
if it satisfies the \emph{Markov property}: for any vertex $i$ with neighbors $\mathcal N(i)$ in $G$ and for $X \sim \nu$, $X_i$ is conditionally
independent of $X_{\sim i}$ given $X_{\mathcal N(i)}$. Such a distribution is also referred to as an \emph{undirected graphical model}, see \cite{lauritzen1996graphical}.
Given a graph $G$, we let $d_G(i,j)$ denote the graph distance
between $i$ and $j$.

\paragraph{Other notation.} Given square matrices $X,Y$ we write $[X,Y] = XY - YX$ for the usual commutator. We write $[X, \cdot]$ to denote the adjoint map $Y \mapsto [X,Y]$. We now come to the important definition of influences for our setting.
\begin{definition} \label{def:inf}
Given a function $f : \Sigma^n \to \{0,1\}$, we define the influence of coordinate $i$ to be
\[ I_i(f) = \Pr_{X \sim \nu}[\exists x'_i, f(X) \ne f(X_1,\ldots,X_{i - 1}, x'_i, X_{i + 1}, \ldots, X_n)]. \]
\end{definition}
We write
$d_H(x,y) = \#\{i : x_i \ne y_i\}$ to denote the usual Hamming metric on $\Sigma^n$.
Given a vector $x \in \Sigma^n$ and $i \in [n]$, $x_{\sim i} \in \Sigma^{n - 1}$ denotes the same vector with coordinate $i$ removed.

\subsection{Main Results}
Our results hold in a very general setting: they apply to all undirected graphical models with
bounded marginals, bounded degree, and which satisfy the log-Sobolev inequality. These assumptions
are formally laid out below. In Section~\ref{sec:example}, we illustrate some of the special cases
where the log-Sobolev inequality is known to hold and give references to others.
\begin{assumption}\label{assumptions}
The probability measure $\nu$ on $\Sigma^n$ for some $n \ge 1$ satisfies that:
\begin{enumerate}
\item There exists a constant $b \ge 1$ such that 
\begin{equation} \label{eqn:mu}
\nu(x)/\nu(y) \in [1/b,b]
\end{equation}
for any $x,y \in \Sigma^n$ with Hamming distance one. In other words, $\nu$ has bounded marginals under pinning.
\item The Glauber dynamics for $\nu$ satisfy the log-Sobolev inequality with constant $\rho \in (0,1]$.
\item The distribution $\nu$ is a Markov random field with respect to a graph $G$ of maximum degree $\Delta$.
\end{enumerate}
\end{assumption}
In our key contribution, we show that these assumptions suffice to prove general versions of Talagrand's theorem and the KKL inequality:
\begin{theorem}[Theorem~\ref{thm:talagrand} below]\label{thm:talagrand-main}
For any $n \ge 1$, $\nu$ satisfying Assumption~\ref{assumptions}, and any $f : \Sigma^n \to \mathbb R$, we have
\begin{equation} \label{eq:main}
\Var_{\nu}(f) \le \frac{Cq^4 b^4 \Delta^2}{\rho} \sum_j \frac{\|L_j f\|_2^2}{1 + \log(\|L_j f\|_2/\|L_j f\|_1)}
\end{equation}
for some absolute constant $C > 0$.
\end{theorem}
\begin{theorem}[Theorem~\ref{thm:kkl} below]\label{thm:kkl-main}
There exists $\alpha_{b,\rho,\Delta,q} > 0$ such that the following is true.
For any $n \ge 1$, $\nu$ satisfying Assumption~\ref{assumptions}, and any $f : \Sigma^n \to \{0,1\}$,
there exists a coordinate $k \in [n]$ such that
\[ I_k(f) \ge \alpha_{b,\rho,\Delta,q} \Var(f)\log(n)/n. \]
\end{theorem}
Both of these results are derived as consequences of a new comparison inequality between
the variance and derivatives of a function $f$ (Theorem~\ref{thm:technical} below).
Our results in (\ref{eq:main}) vastly generalize results of Cordero-Erasquin and Ledoux \cite{CorderoLedoux:12}.
In \cite{CorderoLedoux:12} a statement similar to (\ref{eq:main}) was proven under the assumption that the operators $L_i$ and semigroup $H_t$ `` weakly commute" (equation (15) there). This is valid
for product measures and a few other interesting examples in  \cite{CorderoLedoux:12} such as the symmetric group, the sphere etc.
However, in our setting it fails very badly --- an update at one site affects all of its neighbors, which affects their neighbors, and so on. In our proof we follow \cite{CorderoLedoux:12} in writing the variance as an ``integral over the heat semi-group" (equation
(\ref{eq:dervar}) below). Then, in our main contribution we provide a new analysis for this noncommutative setting which controls the commutators corresponding to all of these interactions.
\subsection{Applications to voting}
There is a long history of using Markov random fields/statistical physics models to model the correlated preferences of voters in elections, for example to estimate the probability of a Condorcet paradox (e.g. \cite{raffaelli2005statistical,columbu2008nature,galam1997rational,gehrlein2006condorcet,koehler2020phase}). Our results have a natural interpretation in the voting context.
If each entry $X \sim \nu$ corresponds to the preference of an individual, and $f : \Sigma^n \to \{0,1\}$ is an election rule which
takes as input these preferences and aggregates them into a choice between two candidates, then our generalized KKL theorem
says that one voter has influence $\Omega(\log(n)/n)$ provided both candidates have a non-negligible chance of winning \emph{a priori}.

What about larger coalitions? Before stating our result, it is natural in the context of elections to assume that voters preferences are also binary valued (i.e. $\Sigma = \{\pm 1\}$) and that the function $f$ is \emph{monotone}, i.e. if $x \le y$ then $f(x) \le f(y)$. Under these assumptions,
the following corollary shows in particular that a coalition of size $\omega(n/\log(n))$ has influence $1 - o(1)$ on a fair election. It follows by iteratively applying our generalization of the KKL theorem, and generalizes
Corollary 3.5 of \cite{KaKaLi:88} where the case of the uniform measure was considered.
\begin{corollary}[Corollary~\ref{corr:coalition} below]
For any $n \ge 1$ and $\nu$ satisfying Assumption~\ref{assumptions}, the following is true.
For any $\epsilon > 0$ and and monotone function $f : \{\pm 1\}^n \to \{0,1\}$ satisfying $\E_{\nu}[f] \ge \epsilon$, there exists a set of coordinates $S \subset [n]$ such that
\[ \E_{X \sim \nu}[f(X_{\sim S}, X_S \to 1)] \ge 1 - \epsilon \]
and
\[ |S| \le \frac{4(1 + b) \log(1/2\epsilon)}{\alpha_{b,\rho,\Delta}} \cdot \frac{n}{\log(n)} \]
where $\alpha_{b,\rho,\Delta} > 0$ is the constant (independent of $n$) from Theorem~\ref{thm:kkl}.
\end{corollary}
Here the notation $\E_{X \sim \nu}[f(X_{\sim S}, X_S \to 1)]$ refers to the expectation of $f(Y)$ where $X$ is drawn from $\mu$ and $Y_i = 1$ for $i \in S$ while $Y_i = X_i$ for $i \notin S$.

\subsection{Comparison to the Results on Phase Transitions for Monotone Measures}
We next compare our results to work by Graham and Grimmett \cite{graham2006influence} and results of
Duminil-Copin Raoufi and Tassion \cite{duminil2019sharp}
who proved a version of the KKL theorem and sharp thresholds for ``monotonic'' measures.
Consider a monotone function $f : \{0,1\}^n \to \{0,1\}$ and a measure $\mu$ on $ \{0,1\}^n$. Recall the definition of influence,
Definition~\ref{def:inf}.
We now define the {\em effect} $e_i(f,\mu)$ of a variable $i$ on $f$ under $\mu$ as
$\Cov_{\mu}[f,x_i] = \E_{\mu}[f x_i] - \E_{\mu}[f] E_{\mu}[x_i]$
(note that this is $p(1-p)$ times the effect as defined in~\cite{HaKaMo:06}).
We note that
\begin{enumerate}
\item
If $\mu$ is the uniform measure and $f$ is monotone then the effect and the influence are the same up to a constant factor. If $\mu$ is a monotone measure in the sense of \cite{graham2006influence} and $f$ is monotone, the size of the effect can be lower bounded by the influence using the FKG inequality (see \cite{graham2006influence}).
\item
The papers~\cite{graham2006influence} and \cite{duminil2019sharp} both prove sharp phase transitions based on the effects.
In~\cite{graham2006influence}, they do so by proving a version of KKL and in~\cite{duminil2019sharp} they do so by generalizing the results of \cite{OSSS:05} using effects.
Interestingly, their results do not require any correlation decay of the measure, so unlike our results they do not require the log-Sobolev inequality. They do require monotonicity of the measure which our results do not.
\end{enumerate}
There are very important differences between the interpretations of effects and influences. (The importance of this difference
was also discussed by Graham and Grimmett \cite{graham2006influence} where they called effects and influences the ``conditional influences'' and ``absolute influences'' respectively.)

To compare influences and effects in a concrete setting,
we consider the finite-volume Ising model with parameter $\beta$ on the square lattice in dimension $d \ge 2$.
In this (classical) setting, the vertices of our graph correspond to the integer elements of $[-L/2,L/2]^d$ where $L \ge 1$
is the sidelength of the box, and the edges $E$ of the graph connect vertices which are neighbors in the square lattice, i.e. which are
Euclidean distance $1$ from each other. Note that there are $n = (L + 1)^d$ many vertices in total.
Given this graph, the ferromagnetic Ising model is the distribution on $\{\pm 1\}^n$ of the form:
\[
\nu(x) \propto \exp \left( \beta \sum_{(i,j) \in E} x_i x_j \right).
\]
Let $\beta_c(d)$ be the critical inverse temperature of the lattice Ising model in dimension $d$ (see e.g. \cite{ellis2006entropy,pisztora1996surface}). Below $\beta_c$ is the high-temperature/subcritical regime and above $\beta_c$ is the low-temperature/supercritical regime of the model. Informally speaking, in the low temperature phase, the model exhibit symmetry breaking, a typical sample from the model lies either in a mostly $+$ phase or in a mostly $-$ phase, and because of this the Glauber dynamics mix torpidly.

Let $f$ be a monotone function from $\{\pm 1\}^n \to \{0,1\}$ with variance $\Omega(1)$.
The results of \cite{graham2006influence} imply that for all $\beta \geq 0$:
\begin{enumerate}
\item
 There exists a variable whose effect is at least $\Omega(\log n / n)$.
\item
There exists a set $S$ consisting of $O(n \ /\log n)$ many variables such that $E[f | X_S = +] = 1-o(1)$.
\end{enumerate}
As we will now illustrate, the analogous results with influences replaces by effects will fail badly due to the aforementioned phase transition in the Ising model.

The log-Sobolev inequality for this measure, see e.g.~\cite{bauerschmidt2022log} allows us to apply our results to deduce that for $\beta < \beta_c$, i.e. in the \emph{subcritical regime} of the model, we have that:
\begin{itemize}
\item There exists a variable whose influence is at least $\Omega(\log n / n)$.
\item There exists a set $S$ consisting of $O(n \ /\log n)$ many variables such that $E[f(X_{-S},X_S \to 1)] = 1-o(1)$.
\end{itemize}

On the other hand, when $\beta > \beta_c$, i.e. in the \emph{supercritical regime}, it immediately follows from rigorous results on the large deviations of the magnetization in the Ising model \cite{pisztora1996surface,bodineau2003slab} that:
\begin{itemize}
\item For every $i$, the effect of $X_i$ is $\Theta(1)$.
\item For every $i$, the influence of $X_i$ is $\exp(-\Theta(L^{d - 1}))$.
\item For a uniformly random set $S$ with $|S| = \omega(1)$ it holds that $E[f | X_S = +] = 1-o(1)$.
\item For every set $S$ with $|S| = o(n)$ it holds that $E[f(X_{-S},X_S \to 1)] = 0.5 + \exp(-\Theta(L^{d - 1}))$.
\end{itemize}
This shows that our results cannot be proven without assuming correlation decay.


Intuitively, for non-product measures there is a dramatic difference between fixing a variable and conditioning on a variable, as conditioning on a variable changes the measure and therefore changes all other variables.
This shows that our results and the results of GC and DCRT and incomparable. In the setting where both our results and theirs apply (monotone measures which satisfy Assumption~\ref{assumptions}), our versions of Talagrand and KKL are stronger since the influences lower bound the effects.

\section{Proof of Main Results}
In this section, we prove all of our results. It was observed by Cordero-Erasquin and Ledoux \cite{CorderoLedoux:12}
that Talagrand's inequality (and then KKL) can be deduced from an estimate of the form \eqref{eqn:basically-done} below. The most important contribution of our work is to prove this estimate (Theorem~\ref{thm:technical}) in our very general setting, which we do in Section~\ref{sec:technical} below. Given this estimate, we derive the generalized Talagrand's inequality and KKL in Section~\ref{sec:corollaries}, and then show how to obtain the consequences for coalitions in Section~\ref{sec:coalition}.

\subsection{Main functional inequality}\label{sec:technical}
The following is the main technical claim which implies Talagrand's inequality and KKL.
\begin{theorem}\label{thm:technical}
There exists absolute constants $c,c' > 0$ such that the following is true.
For any $\nu$ satisfying Assumption~\ref{assumptions}, $f : \Sigma^n \to \mathbb R$, and for any positive $T \le c/b^2q^2\Delta^2$,
\begin{equation} \Var_{\nu}(f) \le \frac{c' q^2 b^2}{1 - e^{-\rho T}} \int_0^T \sum_{j = 1}^n \|L_j f\|_{1 + e^{-2\rho t}}^2\, dt.
\label{eqn:basically-done}
\end{equation}
\end{theorem}
\begin{proof}
Since the log-Sobolev inequality implies the Poincare inequality, we have that for any $T \ge 0$
\[ \Var(f) = \Var(f) - \Var(H_T f) + \Var(H_T f) \le \Var(f) - \Var(H_T f) + e^{-\rho T} \Var(F) \]
and so
\[ \Var(f) \le \frac{1}{1 - e^{-\rho T}} [\Var(f) - \Var(H_T f)]. \]
To upper bound $\Var(f)$,
it thereby suffices to upper bound for some $T > 0$ the quantity 
\begin{equation} \label{eq:dervar}
\Var(f) - \Var(H_T f) = \int_0^T \mathcal{E}(H_t f, H_t f) dt = \sum_i \int_0^T \E(L_i H_t f)^2 dt.
\end{equation}
The first equality in the equation above holds for any Markov semigroup as proven in~\cite{CorderoLedoux:12}.

We recall the following fact, sometimes called the Hadamard or Baker-Hausdorff Lemma:
\begin{lemma}[Proposition 3.35 of \cite{hall2013lie}]\label{lem:hadamard}
For square matrices $X,Y$, we have
$e^X Y e^{-X} = e^{[X,\cdot]} Y$.
\end{lemma}
The following lemma computes the effect of commuting $L_i$ and $H_T$.
\begin{lemma}
For any $T \ge 0$ and $i \in [n]$ we have
$L_i H_T = H_T M_{T,i}$
where
\begin{equation} \label{eq:M}
M_{T,i} := \sum_{k = 0}^{\infty} \frac{T^k}{k!} \sum_{(j_1,\ldots,j_k) \in \mathcal S_{k,i}} [ \cdots [[[P_i,P_{j_1}], P_{j_2}], P_{j_3}] \cdots P_{j_{k}}]. \end{equation}
Here
\begin{equation}\label{eqn:S}
\mathcal S_{k,i} := \{(j_1,\ldots,j_k) : j_i \in \mathcal{N}^+(\{i,j_1, \ldots, j_{i - 1}\})\}  \end{equation} and $\mathcal{N}^+(U)$ denotes the union of $U$ and the neighbors of nodes $U$ in the graph.
\end{lemma}
\begin{proof}
Note that by applying Lemma~\ref{lem:hadamard} to a negated matrix $X$, we have the identity for square matrices $X,Y$
\[ e^{-X} Y e^{X} = e^{[\cdot, X]} Y. \]
 Since $H_t = e^{tL}$, we therefore get
\[ H_T^{-1} L_i H_T = \sum_{k = 0}^{\infty} \frac{T^k}{k!} [L_i,L]^{(k)} \]
where $[L_i,L]^{(k)}$ denotes the iterated commutator of the following form: $[L_i,L]^{(0)} = L_i$ and $[L_i,L]^{(k)} = [[L_i,L]^{(k - 1)},L]$. 

To compute the commutator, first observe
\[ [L_i,L] = \sum_{j : i \sim j} [P_i,P_j] \]
since $L_i = P_i - I$ and $P_i$ commutes with $P_j$ when $i \not \sim j$. For the same reason, we have more generally that
\[ [L_i,L]^{(k)} = \sum_{(j_1,\ldots,j_k) \in \mathcal S_{k,i}} [ \cdots [[[P_i,P_{j_1}], P_{j_2}], P_{j_3}], \cdots P_{j_k}] \]
which proves the result.
\end{proof}
\begin{lemma}\label{lem:s-size}
With the notation of \eqref{eqn:S}, $|\mathcal S_{k,i}| \le (\Delta + 1)^k k^k$ for any $i,k$.
\end{lemma}
\begin{proof}
Observe that we can encode $j_k$ as an element of $[k] \times [\Delta + 1]$ by choosing one of its predecessors $i,\ldots,j_{k - 1}$ and specifying whether $j_k$ is equal to that node or one of that node's $\Delta$ neighbors. Performing this encoding recursively proves the result.
\end{proof}
Therefore recalling the definition of $M_{t,i}$ in (\ref{eq:M}) to get the first equality and applying hypercontractivity  to get the following inequality we have
\begin{align*}
\MoveEqLeft \int_0^T \|L_i H_t f\|_{2}^2 dt  \\
&= \int_0^T \|H_t M_{t,i}f\|_{2}^2 dt \\
&\le \int_0^T \|M_{t,i}f\|^2_{1 + e^{-2\rho t}} dt \\
&= \int_0^T \left\|\sum_{k = 0}^{\infty} \frac{t^k}{k!} \sum_{(j_1,\ldots,j_k) \in \mathcal{S}_{i,k}} [ \cdots [[[P_i,P_{j_1}], P_{j_2}], P_{j_3}], \cdots, P_{j_k}] f \right\|^2_{1 + e^{-2\rho t}} dt \\
&\le \int_0^T \left(\sum_{k = 0}^{\infty} \frac{t^k}{k!} \sum_{(j_1,\ldots,j_k) \in \mathcal{S}_{i,k}} \left\|[ \cdots [[[P_i,P_{j_1}], P_{j_2}], P_{j_3}], \cdots, P_{j_k}] f \right\|_{1 + e^{-2\rho t}}\right)^2 dt \\
&\le \left(\sum_{k = 0}^{\infty} \frac{T^k}{k!} (\Delta + 1)^k k^k\right) \int_0^T \sum_{k = 0}^{\infty} \frac{t^k}{k!} \sum_{(j_1,\ldots,j_k) \in \mathcal{S}_{i,k}} \left\|[ \cdots [[[P_i,P_{j_1}], P_{j_2}], P_{j_3}], \cdots, P_{j_k}] f \right\|_{1 + e^{-2\rho t}}^2 dt \\
&\le 2\int_0^T \sum_{k = 0}^{\infty} \frac{t^k}{k!} \sum_{(j_1,\ldots,j_k) \in \mathcal{S}_{i,k}} \left\|[ \cdots [[[P_i,P_{j_1}], P_{j_2}], P_{j_3}], \cdots, P_{j_k}] f \right\|_{1 + e^{-2\rho t}}^2 dt
\end{align*}
where we used the triangle inequality, in the second-to-last step we applied the Cauchy-Schwarz inequality and Lemma~\ref{lem:s-size}, and in the last step we used
the assumption that $T$ is small compared to $1/\Delta^2$.
\begin{lemma}\label{lem:term-bound}
For any $p \ge 1$, $i \in [n]$, $k \ge 0$ and for $\mathcal S_{i,k}$ as defined in \eqref{eqn:S}, we have
\[ \sum_{(j_1,\ldots,j_k) \in \mathcal{S}_{i,k}} \left\|[ \cdots [[[P_i,P_{j_1}], P_{j_2}], P_{j_3}], \cdots P_{j_k}] f \right\|^2_{p} \le 2(\Delta + 1)^k (k + 1)^{k + 4} (2qb)^{2k + 2} \max_{j : d_G(j,i) \le k} \|L_{j} f\|_p^2  \]
\end{lemma}
\begin{proof}
For notational convenience, define $j_0 = i$.
Observe that
\begin{align*}
&[ \cdots [[[P_i,P_{j_1}], P_{j_2}], P_{j_3}], \cdots, P_{j_k}]  \\
&= P_{j_k} [ \cdots[[[P_i,P_{j_1}], P_{j_2}], P_{j_3}], \cdots, P_{j_{k - 1}}] - [ \cdots[[[P_i,P_{j_1}], P_{j_2}], P_{j_3}], \cdots, P_{j_{k - 1}}]  P_{j_k} \\
&= \sum_{\alpha} (- 1)^{r(\alpha)} \left(P_{j_k} P_{\alpha} - P_{\alpha} P_{j_k}\right)
\end{align*}
where $\alpha$ ranges over a subset of permutations of $(j_0,\ldots,j_{k - 1})$ of size at most $2^{k}$ that arise when expanding out the iterated commutator, and $r(\alpha) \in \{0,1\}$ encodes the corresponding sign of this term.
Let
\[ K_{j_0,\ldots,j_k}(x) = \{ y : y_{\sim \{j_0,\ldots,j_k\}} = x_{\sim \{j_0, \ldots, j_k\}}\}  \]
denote the set of spin configurations which disagree with $x$ only within $\{i,j_1,\ldots,j_k\}$.
Using that the dynamics only update sites $j_0,\ldots,j_k$ and using the triangle inequality we have that
\begin{align*}
|([P_{j_k} P_{\alpha} - P_{\alpha} P_{j_k}]f)(x)| &\le \max_{y,y' \in K_{i,j_1,\ldots,j_k}(x)}  |f(y) - f(y')| \\
&\le (k + 1) \max_{z,z' \in K_{i,j_1,\ldots,j_k}(x)  : d_H(z,z') = 1} |f(z) - f(z')|.
\end{align*}
Hence taking the average over $x$, we find
\begin{align*}
\MoveEqLeft
\sum_x \nu(x) |([P_{j_k} P_{\alpha} - P_{\alpha} P_{j_k}]f)(x)|^p \\
&\le (k + 1)^p \sum_x \nu(x)  \max_{z,z' \in K_{j_0,\ldots,j_k}(x) : d_H(z,z') = 1} |f(z) - f(z')|^p  \\
&\le 2^p (k + 1)^p \sum_x \nu(x)  \max_{z \in K_{j_0,\ldots,j_k}(x), \ell \in \{j_0,\ldots,j_k\}} |(L_{\ell} f)(z)|^p  \\
&\le 2^p (k + 1)^p \sum_x \nu(x)  \max_{z \in K_{j_0,\ldots,j_k}(x)} (|(L_{j_0} f)(z)| + \cdots + |(L_{j_k} f)(z)|)^p  \\
&\le 2^p (k + 1)^p (qb)^{k + 1} \sum_z \nu(z) (|(L_{j_0} f)(z)| + \cdots + |(L_{j_k} f)(z)|)^p
\end{align*}
where in the second to last step we used Lemma~\ref{lem:diff-bound}, and we arrived at the last step by considering the $z$ which achieves the inner maximum, and used the fact that $\nu(x) \le b^{k + 1} \nu(z)$ and that there are at most $q^{k + 1}$ such $x$ for each $z$.
Hence by the $L_p$ triangle inequality, $p \ge 1$, and $1 \le b$,
\[ \|[P_{j_k} P_{\alpha} - P_{\alpha} P_{j_k}]f\|_p \le 2(k + 1) (qb)^{k + 1} \sum_{r = 0}^k \|L_{j_r} f\|_p \le  2(k + 1)^2 (qb)^{k + 1} \max_{r} \|L_{j_r} f\|_p. \]
 Using that $\alpha$ ranges over a set of size at most $2^{k}$, we find by the $L_p$ triangle inequality
\[ \|[ \cdots [[[P_i,P_{j_1}], P_{j_2}], P_{j_3}], \cdots, P_{j_k}]f\|_p  \le \sum_{\alpha} \|P_{j_k} P_{\alpha} - P_{\alpha} P_{j_k}]\|_p \le 2(k + 1)^2 (2qb)^{k + 1}  \max_{0 \le r \le k} \|L_{j_r} f\|_p   \]
and using Lemma~\ref{lem:s-size} we have
\[ \sum_{(j_1,\ldots,j_k) \in \mathcal{S}_{i,k}} \left\|[ \cdots [[[P_i,P_{j_1}], P_{j_2}], P_{j_3}], \cdots P_{j_k}] f \right\|^2_{p} \le 2(\Delta + 1)^k (k + 1)^{k + 4} (2qb)^{2k + 2} \max_{j : d_G(j,i) \le k} \|L_{j} f\|_p^2  \]
as desired.
\end{proof}
\begin{lemma}\label{lem:diff-bound}
Suppose $\nu$ is a distribution on $\Sigma^n$. 
For any function $f : \Sigma^n \to \mathbb R$, and $y,z \in \Sigma^n$ differing only at site $i$ we have
\[ \frac{1}{2} |f(y) - f(z)| \le \max\{|(L_i f)(y)|,|(L_i f)(z)|\}. \]
For any $x \in \Sigma^n$ we have
\[ |(L_i f)(x)| \le \max_{y,z : x_{\sim i} = y_{\sim i} = z_{\sim i}} |f(y) - f(z)|. \]
\end{lemma}
\begin{proof}
Expanding the definition, we have
\[ (L_i f)(x) = (P_i f)(x) - f(x) = \E[f(X) \mid X_{\sim i} = x_{\sim i}] - f(x) \]
so the latter bound follows immediately, and the former bound follows from the triangle inequality as
\[ |f(y) - f(z)| \le |(L_j f)(y)| + |(L_j f)(z)| \le 2 \max \{|(L_j f)(y)|, |(L_j f)(z)|\}. \]
\end{proof}

Using Lemma~\ref{lem:term-bound}, if $T \le c/q^2 b^2\Delta^2$ for some absolute constant $c > 0$, we have for all $t \le T$ that for some constant $c' > 0$,
\[ \sum_{k = 0}^{\infty} \frac{t^k}{k!} \sum_{(j_1,\ldots,j_k) \in \mathcal{S}_{i,k}} \left\|[ \cdots [[[P_i,P_{j_1}], P_{j_2}], P_{j_3}], \cdots, P_{j_k}] f \right\|^2_{1 + e^{-2\rho t}} \le q^2 b^2 \sum_{j = 1}^n (c'/\Delta)^{d_G(j,i)} \|L_j f\|_{1 + e^{-2\rho t}}^2 \]
and summing over $i$ and using that the number of nodes at exactly distance $k$ from node $j$ is at most $\Delta^k$, this gives
\[ 2\sum_i \sum_{k = 0}^{\infty} \frac{t^k}{k!} \sum_{(j_1,\ldots,j_k) \in \mathcal{S}_{i,k}} \left\|[ \cdots [[[P_i,P_{j_1}], P_{j_2}], P_{j_3}], \cdots, P_{j_k}] f \right\|^2_{1 + e^{-2\rho t}} \le c' q^2 b^2 \sum_{j = 1}^n \|L_j f\|_{1 + e^{-2\rho t}}^2. \]

Hence we have for $T \le c/q^2 b^2\Delta^2$ that
\[ \int_0^T \|L_i H_t f\|_2^2 dt \le c' q^2 b^2 \int_0^T \sum_{j = 1}^n \|L_j f\|_{1 + e^{-2\rho t}}^2\, dt \]
which gives the desired bound
\begin{equation*} \Var(f) \le \frac{c' q^2 b^2}{1 - e^{-\rho T}} \int_0^T \sum_{j = 1}^n \|L_j f\|_{1 + e^{-2\rho t}}^2\, dt.
\end{equation*}
\end{proof}
\subsection{Generalized Talagrand and KKL Inequalities}\label{sec:corollaries}
We now show how to deduce the Talagrand and KKL inequalities from Theorem~\ref{thm:technical}.
The proof of these implications follows from the work of Cordero-Erausquin and Ledoux \cite{sh} and is reproduced for convenience.
The first result generalizes Talagrand's inequality:
\begin{theorem}\label{thm:talagrand}
For any $n \ge 1$, $\nu$ satisfying Assumption~\ref{assumptions}, and any $f : \Sigma^n \to \mathbb R$, we have
\begin{equation} \Var_{\nu}(f) \le \frac{Cq^4 b^4 \Delta^2}{\rho} \sum_j \frac{\|L_j f\|_2^2}{1 + \log(\|L_j f\|_2/\|L_j f\|_1)}
\end{equation}
for some absolute constant $C > 0$.
\end{theorem}
\begin{proof}
Making the change of variables $p = 1 + e^{-2\rho t}$, $dp = -2\rho e^{-2\rho t} dt$ and assuming $T \le 1/2\rho$ we have by Holder's inequality
\[ \int_0^T \|L_j f\|^2_{1 + e^{-2\rho t}} dt \le \frac{2}{\rho} \int_1^2 \|L_j f\|^2_p dp \le \frac{2}{\rho} \|L_j f\|_2^2 \int_1^2 d_j^{2\theta(p)} dp  \]
where $1/p = \theta + (1 - \theta)/2 = (1 + \theta)/2$ and
\[ d_j := \|L_j f\|_1/\|L_j f\|_2 \le 1. \]
Note that
\[ \frac{d\theta}{dp} = -2/p^2 \]
so making the change of variables $s = 2\theta(p)$, $ds = (-4/p^2) dp$ we have
\begin{align*}
\int_1^2 d_j^{2\theta(p)} dp
&\le \int_0^2 d_j^s (p(s)^2/4) ds \\
&\le \int_0^2 d_j^s ds = \frac{1 - d_j^2}{\log(1/d_j)} = \frac{(1 - d_j^2)(1 + 1/\log(1/d_j))}{1 + \log(1/d_j)} \le \frac{2}{1 + \log(1/d_j)}.
\end{align*}
hence
\[ \int_0^T \|L_j f\|^2_{1 + e^{-2\rho t}} dt \le \frac{4}{\rho(1 + \log(1/d_j))}. \]
Combining with Theorem~\ref{thm:technical}, we have for $T = c/q^2 b^2 \Delta^2$ that for some absolute constant $C > 0$
\[ \Var(f) \le \frac{Cq^4 b^4 \Delta^2}{\rho}  \sum_j \frac{\|L_j f\|_2^2}{1 + \log(\|L_j f\|_2/\|L_j f\|_1)} \]
which proves the analogue of Talagrand's inequality.
\end{proof}
Now we generalize KKL:
\begin{theorem}\label{thm:kkl}
There exists $\alpha_{b,\rho,\Delta,q} > 0$ such that the following is true.
For any $n \ge 1$, $\nu$ satisfying Assumption~\ref{assumptions}, and any $f : \Sigma^n \to \{0,1\}$,
there exists a coordinate $k \in [n]$ such that
\[ I_k(f) \ge \alpha_{b,\rho,\Delta,q} \Var(f)\log(n)/n. \]
\end{theorem}
\begin{proof}
By combining Lemma~\ref{lem:influence-derivative} with Theorem~\ref{thm:talagrand} we have that
\begin{equation}\label{eqn:kkl-step} \Var(f) \le C \sum_j \frac{I_j(f)}{1 - \log(bq\sqrt{I_j(f)})}
\end{equation}
where $C = C_{b,\rho,\Delta,q} > 0$.
Fix $b,\rho,\Delta,q$ and
suppose for contradiction that the conclusion of the theorem is false. The conclusion of the theorem is trivially true if $n = 1$, so it must be that for any $\alpha \in [0,1]$ there exists $n \ge 2$, $\nu$ satisfying Assumption~\ref{assumptions}, and $f : \Sigma^n \to \{0,1\}$ so that
\[ I_k(f) \le \alpha \Var(f) \log(n)/n \]
 for all $k \in[n]$. In particular $I_k(f) \le \alpha \log(n)/n$ since $\Var(f) \le 1$.
 Combining with \eqref{eqn:kkl-step} and dividing through by $\Var(f)$, we have
\begin{align*} 1
&\le \frac{C \alpha  \log(n)}{1 - \log(bq\sqrt{\alpha \log(n)/n}))} \\
&= \frac{C \alpha \log n}{1 - \log(bq \alpha^{1/2}) + (1/2) [\log(n) - \log \log(n)]} \\
&= \frac{C \alpha}{1/\log(n) - \log(bq \alpha^{1/2})/\log(n) + (1/2) [1 - [\log \log(n)]/\log(n)]}.
\end{align*}
which is a contradiction for any
\[ \alpha < \min\left\{\frac{1}{b^2q^2}, \frac{1}{C} \inf_{n \ge 2} \left[1/\log(n) + (1/2) [1 - [\log \log(n)]/\log(n)] \right] \right\}. \]
\end{proof}

\begin{lemma}\label{lem:influence-derivative}
For $f : \Sigma^n \to \{0,1\}$ and any $\nu$ satisfying Assumption~\ref{assumptions}, we have for any $p \ge 1$
\[ I_i(f) \ge \E |L_i f|^p \ge \frac{1}{(qb)^p} I_i(f) \]
\end{lemma}
\begin{proof}
Recall that
\[ I_i(f) = \Pr_{X \sim \nu}[\exists x'_i \in \Sigma, f(X) \ne f(X_1,\ldots,X_{i - 1}, x'_i, X_{i + 1}, \ldots, X_n)]. \]
Given $x \in \Sigma^n$, if $f(x) = f(x_1,\ldots,x_{i - 1}, x'_i, x_{i + 1}, \ldots, x_n)$ for all $x'_i \in \Sigma$ then this means that
$(L_i f)(x) = 0$. Since $|L_i f| \le 1$, this implies that
\[ |(L_i f)(x)| \le \bone(\exists x'_i \in \Sigma, f(x) \ne f(x_1,\ldots, x_{i - 1}, x'_i, x_{i + 1}, \ldots, x_n)). )\]
 On the other hand, if there exists some $x'_i$ such that $f(x) \ne f(x_1,\ldots,x_{i - 1}, x'_i, x_{i + 1}, \ldots, x_n)$ then this implies that
\[ |(L_i f)(x)| = |f(x) - \E[f(X) \mid X_{\sim i} = x_{\sim i}]| \ge \Pr(X_i = x'_i \mid X_{\sim i} = x_{\sim i}) \ge 1/qb \]
by \eqref{eqn:mu}. Therefore
\[ \frac{1}{qb}\bone(\exists x'_i \in \Sigma, f(x) \ne f(x_1,\ldots, x_{i - 1}, x'_i, x_{i + 1}, \ldots, x_n)) \le |(L_i f)(x)| \]
 Hence taking expectation over $X$ we have for any $p \ge 1$
\[ I_i(f) \ge \E |L_i f|^p \ge \frac{1}{(qb)^p} I_i(f) \]
as claimed.
\end{proof}
\subsection{Application to coalitions}\label{sec:coalition}
We now discuss the application of our result to the existence of coalitions for monotone
voting rules. In this section, we restrict to the case of $\Sigma = \{\pm 1\}$ and
recall that a function $f : \{\pm 1\}^n \to \mathbb{R}$ is \emph{monotone} if
\[ f(x) \le f(y) \]
for any pair such that $x \le y$  coordinatewise.

The following corollary shows in particular that a coalition of size $\omega(n/\log(n))$ has influence $1 - o(1)$ on a fair election. It follows by iteratively applying our generalization of the KKL theorem, and generalizes
Corollary 3.5 of \cite{KaKaLi:88} where the case of the uniform measure was considered.
\begin{corollary}\label{corr:coalition}
For any $n \ge 1$ and $\nu$ satisfying Assumption~\ref{assumptions}, the following is true.
For any $\epsilon > 0$ and and monotone function $f : \{\pm 1\}^n \to \{0,1\}$ satisfying $\E_{\nu}[f] \ge \epsilon$, there exists a set of coordinates $S \subset [n]$ such that
\[ \E_{X \sim \nu}[f(X_{\sim S}, X_S \to 1)] \ge 1 - \epsilon \]
and
\[ |S| \le \frac{4(1 + b) \log(1/2\epsilon)}{\alpha_{b,\rho,\Delta}} \cdot \frac{n}{\log(n)} \]
where $\alpha_{b,\rho,\Delta} > 0$ is the constant (independent of $n$) from Theorem~\ref{thm:kkl}.
\end{corollary}
\begin{proof}
We construct a sequence of sets $S_0,S_1,\ldots$ iteratively.
Let $S_0 = \{\}$.
For each $t \ge 0$, define $f_t(x) = f(x_{\sim S_t}, 1_{S_t})$, i.e. $f_t$ is the same as $f$ except that
it ignores the input $x_{S_t}$ and replaces it by all-ones. 
Either
\[ \Pr_{X \sim \nu}[f_t = 1] \ge 1 - \epsilon \]
or we define a set $S_{t + 1}$ in the following way. By Theorem~\ref{thm:kkl},
there exists some $k_t \in [n]$ such that
\[ I_{k_t}(f_t) \ge \alpha \Var(f_t) \frac{\log(n)}{n} \]
where $\alpha = \alpha_{b,\rho,\Delta} > 0$ does not depend on $n$,
and we let $S_{t + 1} = S_t \cup k_t$. Now defining $f_{t + 1}(x) = f(x_{\sim S_{t + 1}}, 1_{S_{t + 1}})$,
we have by monotonicity that
\[ \Pr_{\nu}(f_{t + 1} = 1) = \Pr_{\nu}(f_t = 1) + \Pr_{\nu}(f_{t + 1} > f_t). \]
Furthermore,
\begin{align*}
\Pr_{\nu}(f_{t + 1} > f_t)
&= \E_{X \sim \nu}[1(f_{t}(X_{\sim k_t}, 1) > f_t(X))] \\
&=  \E_{X \sim \nu}[1(f_{t}(X_{\sim k_t}, 1) > f_t(X)) \cdot 1(X_{k_t} = -1)] \\
&=  \E_{X \sim \nu}[1(f_{t}(X_{\sim k_t}, 1) > f_{t}(X_{\sim k_t}, -1)) \cdot 1(X_{k_t} = -1)] \\
&= \E_{X \sim \nu}[1(f_{t}(X_{\sim k_t}, 1) > f_{t}(X_{\sim k_t}, -1)) \cdot \Pr(X_{k_t} = -1 \mid X_{\sim k_t})] \\
&\ge \frac{I_{k_t}(f_t)}{1 + b}
\end{align*}
where in the last equality we applied the law of total expectation, and in the final step we used that
\[ \Pr_{X \sim \nu}(X_{k_t} = -1 \mid X_{\sim k_t}) \ge \frac{1}{1 + b} \]
by Assumption~\ref{assumptions}. Therefore, if $p_t = \Pr(f_t = 1)$ we have that
\[ p_{t + 1} \ge p_t + \frac{\alpha}{1 + b} p_t(1 - p_t) \frac{\log(n)}{n}. \]
It follows that if $p_t < 1/2$, $p_{t + 1} \ge (1 + \frac{\alpha \log(n)}{(1 + b)n}) p_t \ge \exp\left(\frac{\alpha \log(n)}{2(1 + b)n}\right) p_t$,
so $p_t > 1/2$ for any $t > \frac{2(1 + b)n}{\alpha \log(n)} \log(1/2\epsilon)$. By a symmetrical argument,
we have that $p_t \ge 1 - \epsilon$ for $t > \frac{4(1 + b)n}{\alpha \log(n)} \log(1/2\epsilon)$.
\end{proof}
\section{Some examples}\label{sec:example}
There is a vast literature establishing log-Sobolev inequalities for spin systems on the hypercube.
For concreteness, we give a few examples of settings where the log-Sobolev constant is known
to be bounded, and as a consequence our results can be applied.

\paragraph{Sparse Markov random field under $\ell_2$-Dobrushin uniqueness condition.} Suppose that $\nu$
is a Markov random field on a graph of maximum degree $\Delta$ with $n$ vertices, and
define the Dobrushin matrix $A \in \mathbb R^{n \times n}$ to have zero diagonal and off-diagonal entries
\[ A_{ij} = \max_{y \in \Sigma^n, z} d_{TV}(\Pr_{\nu}[X_i = \cdot \mid X_{\sim i} = y_{\sim i}], \Pr_{\nu}[X_i = \cdot \mid X_{\sim i,j} = y_{\sim i,j}, X_j = z]). \]
Suppose also that $\nu$ satisfies the $b$-bounded marginal assumption from Assumption~\ref{assumptions}.
Then if $\|A\|_{OP} < 1$, it was shown by Marton
 \cite{marton2019logarithmic} that $\nu$ satisfies the log-Sobolev inequality with log-Sobolev constant polynomial
 in $b$ and $q$.

\paragraph{Special case: Ising under Dobrushin's uniqueness threshold.} As a special case of the above, suppose that
\[ \nu(x) \propto \exp\left(\sum_{(i,j) \in E} J_{ij} x_i x_j + \sum_i h_i x_i\right) \]
is a probability measure on the hypercube $\{\pm 1\}^n$ parameterized by $J,h$ where $E$ is the edge set of a sparse graph of maximum degree $\Delta$. If $\sum_j |J_{ij}| < 1 - \delta$ for all
$i$, and $\sum_i |h_i| < H$, one can directly show from the definition of the model that it is marginally bounded with $b = \exp(O(1 + H))$
and satisfies Dobrushin's uniqueness condition (by applying Gershgorin's disk theorem), hence our result applies.  Note that we do not need any assumption on the sign of the interactions $J_{ij}$ or external field $h_i$.

\paragraph{Additional references.}
There are many settings outside of Dobrushin's uniqueness condition where the log-Sobolev inequality is known. For example, the case of the lattice Ising model we discussed earlier is not contained in this regime. See e.g.\ \cite{stroock1992logarithmic,chen2021optimal,bauerschmidt2019very,eldan2022spectral,bauerschmidt2022log} for a few relevant references. In particular, by the result of Chen, Liu, and Vigoda \cite{chen2021optimal}, the log-Sobolev constant can be bounded purely as a function of $b,\Delta$
and the ``spectral independence'' constant of the distribution $\nu$ --- so our assumption that the log-Sobolev constant is bounded
can be replaced by the assumption of spectral independence.

\paragraph{Acknowledgment}
F.K. was supported in part by NSF award CCF-1704417, NSF award IIS-1908774, and N. Anari’s
Sloan Research Fellowship. D.M. was supported by a Sloan Research Fellowship, NSF CCF award 2227876 and NSF CAREER award 2239160.
E.M. is partially supported by and Vannevar Bush Faculty Fellowship award ONR-N00014-20-1-2826,  ARO MURI W911NF1910217 and a Simons Investigator Award in Mathematics (622132). 

\bibliographystyle{plain}
\bibliography{ref,my,all}

\end{document}